\RequirePackage{fix-cm}
\documentclass[11pt]{article}

\newcommand{\lp}{\left(}
\newcommand{\rp}{\right)}

\usepackage{mypackages}
\usepackage{makeidx}

\usepackage[T1]{fontenc}
\usepackage{lmodern}
\usepackage{titlesec}

\begin{document}

\author{Estepan Ashkarian, Ataleshvara Bhargava, Nicholas Gismondi, and Matthew Novack}

\title{Intermittent singular solutions of the stationary 2D Navier-Stokes equations in sharp Sobolev spaces}

\maketitle

\begin{abstract}
    In this paper we construct non-trivial solutions to the stationary Navier-Stokes equations on the two dimensional torus which lie in $\bigcap_{\epsilon \in (0,1)} L^{2-\epsilon}(\T^2) \cap \dot H^{-\epsilon}(\T^2)$. Due to the fact that our solutions are not square integrable, we must redefine the notion of solution. Our result gives a sharp extension of the recent work of Lemari\'e-Rieusset, \cite{PGLR} who proved a similar result in the space $\dot{H}^{-1} \cap \operatorname{BMO}^{-1}$.  The main new ingredient is the incorporation of intermittency into the construction of the solutions.
\end{abstract}

\section{Introduction}

\subsection{Motivation and background}
In this paper we consider the stationary two-dimensional incompressible Navier-Stokes equations
    \begin{equation}\label{eq: SNSE}\begin{cases}
        \operatorname{div}(u\otimes u)-\Delta u + \nabla p =0\\
        \operatorname{div}(u)=0\end{cases}\end{equation}
posed on the torus for $u(x)\colon \mathbb{T}^{2}\to \R^2$. We will consider mean zero solutions, that is 
\begin{equation*}
    \int_{\T^{2}}u\,dx=0.
\end{equation*}
Let us start by recalling known results on solutions of~\eqref{eq: SNSE}. If $u$ is a smooth solution of~\eqref{eq: SNSE}, then it must be a trivial one. Indeed, we multiply the equation by $u$, integrate by parts, and use $\operatorname{div}( u)=0$ to obtain 
\begin{align*}
    \sum_{i=1}^{2}\|\nabla u^{i}\|_{L^{2}(\T^2)}^{2}&=\int_{\T^{2}}\partial_{j}u^{i}\partial_{j}u^{i}\,dx\\&=-\int_{\T^{2}}u^{i}\partial_{jj}u^{i}\,dx\\
    &=-\int_{\T^{2}}\partial_{j}(u^{i}u^{j})u^{i}\,dx -\int_{\T^{2}}u^{i}\partial_{i}p\,dx\\
    &=0 \, ,
\end{align*}
which implies for mean-zero solutions that $u\equiv 0$.  The same computation can be justified when we assume $u$ is a solution of~\eqref{eq: SNSE} in $L^{p}(\T^{2})$ when $p>2$. In this scenario, $u$ must in fact be smooth, and hence we can apply the previous argument to conclude $u \equiv 0$; see recent work of Lemari\'e-Rieusset~\cite{PGLR} for more details. Moreover, in~\cite{PGLR} it is shown that any solution $u$ of~\eqref{eq: SNSE} in the Lorentz space $L^{2,1}(\T^{2})$ is trivial.

On the other hand, in~\cite{PGLR} Lemarié-Rieusset proved that there exist nontrivial stationary solutions in $\dot H^{-1}(\T^2)$ to~\eqref{eq: SNSE}.  While it is clear how to define weak solutions to~\eqref{eq: SNSE} if $u \in L^2(\T^2)$, it is not obvious what is meant by a weak solution if $u \otimes u$ is not locally integrable. Lemarié-Rieusset treats this issue by expanding $u \otimes u$ in a doubly infinite Fourier series and checking that this infinite sum converges in a negative Sobolev space.  We shall adopt a similar notion in Definition~\ref{def:para:solns} using paraproducts.  While the result in~\cite{PGLR} is the first for nontrivial stationary weak solutions, there are two earlier but relevant results of Christ~\cite{Christ1, Christ2}.  In~\cite{Christ1}, Christ constructs nontrivial time-dependent solutions of the one dimensional periodic cubic nonlinear Schr\"odinger equation which belong to $C^0_t H^{-\epsilon}_x$.  The key issue is again the interpretation of the cubic nonlinearity $|u|^2 u$ when $u \notin L^3_{t,x}$, and Christ shows that this object has an interpretation as a limit $\lim_{\epsilon\rightarrow 0}|u_\epsilon|^2 u_\epsilon$, where $u_\epsilon$ is a mollified version of $u$. In~\cite{Christ2}, Christ then treats the time-dependent two-dimensional Navier-Stokes equation using similar methodology.  We remark that Christ identifies a connection between his work and the well-known earlier works of Scheffer~\cite{Scheffer93} and Shnirelman~\cite{Shnirelman00} on non-uniqueness for the Euler equations; in particular, we quote the following observation from~\cite{Christ2}:

\begin{center}
\textit{Our construction and that of Shnirelman have in common both the use of
driving forces tending weakly to zero, and the exploitation of this reverse energy cascade.}
\end{center}

\noindent In particular this also highlights the connection between Christ's works~\cite{Christ1, Christ2} and modern convex integration/Nash iteration techniques descended from~\cite{Scheffer93, Shnirelman00}.

\subsection{Main result} In this paper, we consider the following notion of stationary solution. For $f \in H^s$ and $g \in H^{s'}$, $s,s' \in \R$, we formally define
$$
fg = \sum_{j,j' \geq 0} \mathbb{P}_{2^j}f \mathbb{P}_{2^{j'}}g.
$$
The projection operators $\mathbb{P}_{2^j}$ are the standard Littlewood-Paley projections onto frequencies of size $\approx 2^{j}$; see Definition~\ref{def:projs}. There is not much that can be said in general about the convergence of this sum in negative Sobolev spaces. However, we may offer the following definition for the mean-zero portion $\mathbb{P}_{\neq 0}(fg) = (\operatorname{Id}-\fint_{\T^2})(fg)$ of a formal product $fg$.

\begin{definition}[\textbf{Paraproducts in $\dot H^s(\T^2)$}]\label{def:paras}
    Let $f,g$ be distributions, so that $\mathbb{P}_{2^j}(f), \mathbb{P}_{2^{j'}}(g)$ are well-defined for $j,j'\geq 0$.  We say that $\mathbb{P}_{\not = 0}(fg)$ is well-defined as a paraproduct in $\dot{H}^s(\T^2)$ for some $s \in \R$ if
    $$
    \sum_{j,j' \geq 0} \left\Vert \mathbb{P}_{\not = 0}\left(\mathbb{P}_{2^j}(f) \mathbb{P}_{2^{j'}}(g)\right) \right\Vert_{\dot{H}^s} < \infty \, .
    $$
    Then we define
    $$
    \mathbb{P}_{\not = 0}(fg) = \sum_{j,j' \geq 0} \mathbb{P}_{\not = 0} \left(\mathbb{P}_{2^j}(f) \mathbb{P}_{2^{j'}}(g)\right) \, ,
    $$
    since the right-hand side is an absolutely summable series in $\dot H^s(\T^2)$.
\end{definition}

With the above definition in hand, we define a weak notion of solution to~\eqref{eq: SNSE}.

\begin{definition}[\textbf{Weak paraproduct solutions to~\eqref{eq: SNSE}}]\label{def:para:solns}
If $u \in \dot{H}^s$, $s < 0$, we say $u$ is a weak paraproduct solution to~\eqref{eq: SNSE} if there is $s' \in \R$ such that $\mathbb{P}_{\not = 0}(u \otimes u)$ is well defined as a paraproduct in $\dot{H}^{s'}$ in the sense of Definition~\ref{def:paras} and
    $$
    \left\langle \Delta \phi^i, u^i \right\rangle_{\dot{H}^{-s},\dot{H}^s} + \left\langle \partial_j \phi^i, \mathbb{P}_{\not= 0}(u^iu^j)\right\rangle_{\dot{H}^{-s'}, \dot{H}^{s'}} = 0
    $$
    for all smooth, divergence free vector fields $\phi$.
\end{definition}

Our main result is then as follows.

\begin{theorem}[\textbf{Non-trivial solutions of 2D stationary Navier-Stokes equations}]\label{thm:main} There exist nontrivial (non-constant and non-zero) weak solutions to~\eqref{eq: SNSE} in the sense of Definition~\ref{def:para:solns} belonging to $\displaystyle \bigcap_{\epsilon \in (0,1)} L^{2-\epsilon}(\T^2) \cap \dot H^{-\epsilon}(\T^2)$.
\end{theorem} 




\noindent Considering that stationary mean-zero solutions in $L^{2,1}(\T^2)$ must be trivial from~\cite{PGLR}, Theorem~\ref{thm:main} is essentially sharp.  The only possible strengthening would be to address the endpoint case $u \in L^2(\T^2)$.  It appears that there are clear obstructions which prevent any simple extension of our method to the case $u \in L^2(\T^2)$.  However, it is conceivable to us that stationary solutions in $L^2(\T^2)$ may be nontrivial, although a different method of proof may be required.

\begin{remark}[\textbf{Extensions of Theorem~\ref{thm:main}}]
The same techniques we use to prove Theorem~\ref{thm:main} can easily be extended to prove the following: let $u_0$ be any smooth divergence-free vector field.  Then for any $\epsilon>0$, there exists $u_\epsilon,p_\epsilon$ which solves~\eqref{eq: SNSE} in the sense of Definition~\ref{def:para:solns} and satisfies $\| u_\epsilon - u_0 \|_{\dot H^{-\epsilon}} + \| u_\epsilon - u_0 \|_{L^{2-\epsilon}}<\epsilon$.
\end{remark}

\begin{remark}[\textbf{Regularity in $L^\infty$-based Besov spaces}]
Our construction gives solutions lying in $\bigcap_{\epsilon>0} B^{-\sfrac 12 - \epsilon}_{\infty,\infty}$.  Indeed this is immediate from Bernstein's inequality and the absolute summability of the series used to define $u = \sum_{q\geq 0} u_{q+1} - u_{q}$ in $\bigcap_{\epsilon>0} L^{2-\epsilon}$; see~\eqref{e:w:est}.
\end{remark}

\subsection{Related literature}
In addition to the recent work of Lemarié-Rieusset~\cite{PGLR} and the earlier work of Christ~\cite{Christ2}, the past decade has seen numerous flexibility results for the Navier-Stokes equations.  Many of these fall under the umbrella of convex integration/Nash iteration, initiated by De Lellis and Sz\'ekelyhidi in~\cite{DLS09, DLS13} for the Euler equations.  Most relevant to our setting is the intermittent style of Nash iteration, developed first by Buckmaster and Vicol~\cite{BV19} for the Navier-Stokes equations.  Variants of intermittent convex integration have been utilized by Luo~\cite{Luo} for the 4D NSE, Buckmaster, Colombo, and Vicol~\cite{BCV} to concentrate the temporal sets on which the solutions to 3D NSE are rough, and Cheskidov and Luo~\cite{CheskidovLuo, CL2} for solutions in $L^{2-}_t L^\infty_x$ and  $C^t L^p_x$ for $p<2$. Alternatively, Albritton, Bru\'e, and Colombo~\cite{ABC} utilized a spectral instability to prove non-uniqueness of Leray solutions for the forced 3D Navier-Stokes equations, and Albritton and Colombo~\cite{AC} used similar methods to treat the 2D hypo-dissipative Navier-Stokes equations. Finally, we mention recent work of Coicolescu and Palasek~\cite{CP}, in which an alternative method based on a careful choice of a lacunary series as initial data gives non-uniqueness of smooth solutions from critical data.

\subsection{Quick idea of the proof}

We use a Nash iteration method, constructing a nontrivial solution to \eqref{eq: SNSE} as a limit of smooth solutions of the Euler Reynolds system
\begin{subequations}\label{eq:RS}
\begin{align}
\operatorname{div}(u_q \otimes u_q) - \Delta u_q + \nabla p_q = \operatorname{div}{R_q}
\\ \operatorname{div}(u_q)=0 \,  ,
\end{align}
\end{subequations}
where $R_q$ is assumed to be symmetric. We assume $(u_q, R_q)$ satisfy natural inductive assumptions which imply that $u_q \rightarrow u$ in $L^{2-\epsilon} \cap \dot H^{-\epsilon}$.  In particular, we do \emph{not} prove that $R_q\rightarrow 0$ in $L^1$, only that $R_q\rightarrow 0$ in $\dot H^{-2}$; see \ref{i:ind:1}-\ref{i:ind:5} for precise formulations. 

In standard Nash iterations, convergence in the limit as $q\rightarrow \infty$ of $u_q \otimes u_q$ follows from the absolute summability of $\sum_{q\geq 0} \| u_{q+1} - u_q \|_{L^2}$. Since our solutions do not belong to $L^2$, we instead prove that $u_q\otimes u_q$ converges to $u\otimes u$ in the sense of Definition~\ref{def:paras}. Morally speaking, $u_q - u_{q-1} = \mathbb{P}_{2^{f(q)}} u$, where $f(q)\rightarrow \infty$ as $q\rightarrow \infty$ superlinearly. This allows for successive iterates to be decorrelated and affords us stronger control over their products.  We use this to iteratively control the norms of $u_q$ and its Littlewood-Paley projections in negative Sobolev spaces and $L^p$ for $p < 2$ by carefully constructing the velocity increments; see Definition \ref{def:increm}.  Analysis involving the frequency locations of the terms present in the velocity increment plays an important role in proving these estimates. We also control the support sizes and orthogonality properties of the increments.  

To be more concrete, let us write $u = \sum_{q \geq 0} w_q$; then we formally have that 
\begin{align*} u \otimes u & = \sum_{q,q' \geq 0} w_q \otimes w_{q'} 
\\ & = \sum_{q \neq q'} w_q \otimes w_{q'} + \sum_q w_q \otimes w_q.
\end{align*}
For the first sum, we may assume without loss of generality that $q > q'$, and bound the terms by $\| w_{q'} \|_{L^1} \| w_q \|_{L^{\infty}} \lesssim 2^{-(q+q')}$. For the second sum, i.e. ``the diagonal,'' we decompose the terms by writing $w_q \otimes w_q = -R_{q-1} + E$, where $R_{q-1}$ is the previous Reynolds stress error and $E$ is a high frequency oscillation error. Both of these terms will approach zero in $\dot{H}^{-2}$ as $q \rightarrow \infty$, ensuring that the terms on the diagonal converge.  Hence the convergence of $R_q$ to zero in $\dot H^{-2}$ corresponds precisely to the convergence of $u_q \otimes u_q$ in $\dot H^{-2}$.

\subsection{Outline of the paper}
In section~\ref{Section2}, we begin by recalling the basic Littlewood-Paley theory we need to formulate definitions~\ref{def:paras} and \ref{def:para:solns}. We also recall some technical lemmas which will be useful in the sections that follow. In section~\ref{Section3} we formulate our inductive proposition Proposition~\ref{prop:ind} and use it to prove Theorem~\ref{thm:main}. The first portion of section~\ref{Section4} is spent constructing the velocity increment, which will be used to construct the Nash iterates. Then once this is completed, the rest of section~\ref{Section4} is dedicated to the proof of Proposition~\ref{prop:ind}.

\subsection{Acknowledgements}
E.A. was supported by the NSF through grant DMS-2153915.
A.B. acknowledges partial support from M.N. and Trevor Wooley. 
N.G. was supported by the NSF through grant DMS-2400238. M.N. was supported by the NSF through grant DMS-2307357.

\section{Background Theory}\label{Section2}

The following material can be found in \cite{BCD} and \cite{Grafakos}.

\begin{definition}[\textbf{Littlewood-Paley projectors}]\label{def:projs}
    There exists $\varphi \colon \R^{2}\to[0,1]$, smooth, radially symmetric, and compactly supported in $\{6/7 \leq |\xi|\leq 2\}$ such that $\varphi(\xi) = 1$ on $\{1 \leq |\xi| \leq 12/7\}$, 
    \begin{equation}
        \sum_{j\geq 0}\varphi(2^{-j}\xi)=1 \hspace{0.25cm} \text{ for all } \hspace{0.25cm} |\xi|\geq 1, \notag
    \end{equation}
    and $\operatorname{supp}\varphi_{j}\cap \operatorname{supp} \varphi_{j'}=\emptyset$ for all $|j-j'|\geq 2$, where $\varphi_j(\cdot) = \varphi(2^{-j}\cdot)$. We define the projection of a function $f$ on its $0$-mode by 
    \begin{equation}
        \mathbb{P}_{0}f=\fint_{\T^2} f, \notag 
    \end{equation}
    and the projection on the $j^{\rm th}$ shell by
    \begin{equation}
        \mathbb{P}_{2^j}(f)(x)=\sum_{k\in \Z^2}\hat{f}(k)\varphi_{j}(k)e^{2\pi ik\cdot x} \, . \notag 
    \end{equation}
    We also define $\mathbb{P}_{\neq 0}f:=(\operatorname{Id}-\mathbb{P}_{0})f$, and 
    \begin{equation}
        \mathbb{P}_{\leq 2^j}(f)(x) = \sum_{k \in \mathbb{Z}^2} \hat{f}(k)\psi_j(k)e^{2 \pi ik\cdot x}, \notag  
    \end{equation}
    where $\psi_j(\cdot) = \sum_{|j'| \leq j} \varphi_{j'}(\cdot)$.
\end{definition}

\begin{definition}[\textbf{$\dot H^s$ Sobolev spaces}] For $s \in \mathbb{R}$, we define  
$$\dot{H}^{s}(\mathbb{T}^{2})=\left\{ f: \sum_{\substack{k\in \Z^{2}\\k \not = 0}}|k|^{2s}|\hat{f}(k)|^2<\infty\right\}  $$
with the norm induced by the sum above.
\end{definition}
\begin{remark}
For every $f\in \dot H^{s}$ for some $s\in \R$, we can define the Fourier coefficients
\[\hat{f}(k)=\int_{\T^{2}}e^{-2\pi ik\cdot x}f(x)dx, \hspace{0.25cm} \text{ where } \hspace{0.25cm} \T^{2}=[0,1]^{2},\]
and we can define $\mathbb{P}_{2^j}f$ for $j\geq 0$.  Note that each $\mathbb{P}_{2^j}f$ is smooth if $f\in \dot H^{s}$ irrespective of the value of $s\in \R$.
\end{remark} 

\begin{lemma}[\textbf{$L^p$ boundedness of projection operators}]\label{lem:proj}
If $\lambda = 2^j$, then $\mathbb{P}_{\leq \lambda}$ is a bounded operator from $L^p$ to $L^p$ for $1 \leq p \leq \infty$ with operator norm independent of $\lambda$ and $j$. 
\end{lemma}

Recall the following geometric lemma from~\cite[Lemma~4.2]{BSV}, which we shall need to construct the velocity increments.
\begin{lemma}[\textbf{Reconstruction of symmetric tensors}]\label{lem:geom}
     Let $B(I,\epsilon)$ be the ball of radius $\epsilon$ around the identity matrix in the space of $2\times2$ symmetric matrices. We can choose $\epsilon > 0$ such that there exists a finite set $\Lambda  \subset \mathbb{S}^{1}$ and smooth positive functions $\gamma_k \in C^{\infty}(B(I,\epsilon))$ for $k \in \Lambda$ such that the following hold: 
\begin{enumerate}
    \item $5\Lambda \subseteq \mathbb{Z}^2$ 
    \item For all $R \in B(I,\epsilon)$ we have 
    \[ R = \frac{1}{2} \sum_{k \in \Lambda} (\gamma_k(R))^2 k^{\perp} \otimes k^{\perp}. \]
\end{enumerate}
\end{lemma}
In the proof of the above lemma, one chooses $k_1 = (1,0)$, $k_2 = (3/5,4/5)$, and $k_3 = (3/5,-4/5)$. The proof proceeds by writing the identity as a positive linear combination of these three tensors and then applying the implicit function theorem.

\begin{lemma}[\textbf{High-low products in negative Sobolev norms}]\label{lem:RL2}
    Let $a(\cdot)$ be a smooth function on the torus $\T^{2}$ and $V\colon \T^{2}\to \R$ be a mean zero and smooth. Then
    \begin{equation}\label{eq:RL2norm}
        \lim_{\lambda \to \infty}\|a(\cdot)V(\lambda \cdot )\|_{\dot{H}^{-2}}=0.
    \end{equation}
\end{lemma}

\begin{proof}
Consider the following decomposition of the product $a(\cdot)V(\lambda \cdot)$ into low and high frequencies:
\begin{equation}\label{eq:high-low-decomp}
    a(x)V(\lambda x)= \mathbb{P}_{\leq  \frac{\lambda}{2}}(a)(x)V(\lambda x) + \mathbb{P}_{> \frac{\lambda}{2}}(a)(x)V(\lambda x).
\end{equation}
Observe that the low frequency term in \eqref{eq:high-low-decomp} can be dealt with in the following manner:
\begin{align*}
    \|\mathbb{P}_{\leq \frac{\lambda}{2}}(a)V(\lambda \cdot)\|_{\dot{H}^{-2}}^{2}&=\sum_{k\neq 0}|k|^{-4}\left|\lp \mathbb{P}_{\leq \frac{\lambda}{2}}(a)V(\lambda \cdot)\rp^{\wedge}(k) \right|^2\\
    &=\sum_{|k|\gtrsim \frac{\lambda}{2}}|k|^{-4}\left| \lp \mathbb{P}_{\leq \frac{\lambda}{2}}(a)V(\lambda \cdot)\rp^{\wedge}(k) \right|^2\\
    &\leq \sum_{|k|\gtrsim \frac{\lambda}{2}}|k|^{-4}\|\mathbb{P}_{\leq \frac{\lambda}{2}}(a)\|_{L^{\infty}}^{2}\|V(\lambda \cdot)\|_{L^{\infty}}^{2}\\
    &\leq \sum_{|k|\gtrsim \frac{\lambda}{2}}|k|^{-4}\|a\|_{L^{\infty}}^{2}\|V\|_{L^{\infty}}^{2}\\
    &\lesssim \sum_{|k|\gtrsim \frac{\lambda}{2}}|k|^{-4}
\end{align*}
where in the second equality we have used that the lowest frequency present in $V(\lambda \cdot)$ is of order $\lambda$ (since $\hat{V}(0) = 0$) and thus $(\mathbb{P}_{\leq \lambda/2}(a)V(\lambda \cdot))^{\wedge}(k)$ is supported outside the ball of radius comparable to $\lambda/2$. Now upon sending $\lambda \to \infty$ we see that this quantity goes to $0$.  The high frequency term in \eqref{eq:high-low-decomp} on the other hand may be dealt with as follows:
\begin{align*}
\|\mathbb{P}_{>\frac{\lambda}{2}}(a)V(\lambda \cdot)\|^2_{\dot{H}^{-2}}&=\sum_{k\neq 0} |k|^{-4}\left|\lp \mathbb{P}_{>\frac{\lambda}{2}}(a) V(\lambda \cdot) \rp^{\wedge}(k) \right|^{2}\\
&\lesssim \Vert \mathbb{P}_{>\frac{\lambda}{2}}(a) V(\lambda \cdot) \Vert_{L^\infty}^2 \sum_{k\neq 0}|k|^{-4}\\
&\simeq \Vert \mathbb{P}_{>\frac{\lambda}{2}}(a) V(\lambda \cdot) \Vert_{L^\infty}^2\\
\end{align*}
Notice that 
    \[\|\mathbb{P}_{>\frac{\lambda}{2}}(a)V(\lambda \cdot)\|_{L^{\infty}}\leq \|\mathbb{P}_{>\frac{\lambda}{2}}(a)\|_{L^{\infty}}\|V\|_{L^{\infty}} \longrightarrow 0 \quad \text{as} \quad \lambda \to \infty.\]
    Indeed, by the dominated convergence theorem we have 
    \[ \Vert \mathbb{P}_{> \frac{\lambda}{2}}(a) \Vert_{L^{\infty}} \leq \sum_{|k|\gtrsim\frac{\lambda}{2}}|\widehat{a}(k)|  \longrightarrow 0 \quad \text{as} \quad \lambda \to \infty\]
    since $\hat{a} \in \ell^1$. Therefore,
    \[\Vert \mathbb{P}_{>\frac{\lambda}{2}}(a)V(\lambda \cdot) \Vert_{\dot{H}^{-2}} \lesssim \|\mathbb{P}_{>\frac{\lambda}{2}}(a)V(\lambda \cdot)\|_{L^{\infty}}\longrightarrow 0 \, , \]
    and so in particular we get that
    $$
    \Vert a(\cdot) V(\lambda \cdot) \Vert_{\dot{H}^{-2}} \to 0
    $$
    as $\lambda \to \infty$.   
\end{proof}

\begin{lemma}[\textbf{High-higher-low products in negative Sobolev spaces}]\label{lem:RL}
Let $\alpha$ and $V$ be smooth functions on the torus such that $V$ is mean-zero, and let $\{\beta_\lambda\}_{\lambda \in \mathbb{N}}$ be a sequence of smooth functions such that $\sup_{\lambda} \| \beta_\lambda \|_{L^1} < \infty$ and $\lim_{\lambda \rightarrow \infty} \| \mathbb{P}_{> \lambda^2} \beta_\lambda \|_{L^1} \rightarrow 0$. Then we have
    \begin{equation}
        \lim_{\lambda \to \infty}\left\Vert \alpha(\cdot) \beta_{\lambda}(\cdot) V(\lambda^6\cdot) \right\Vert_{\dot{H}^{-2}}=0. \notag 
    \end{equation}
\end{lemma}
\begin{proof}
The result is a variant of the previous lemma, and so we only outline the main steps of the proof. We decompose the product as follows:
\begin{align}
    \alpha(x) \beta_{\lambda}(x) V(\lambda^6x) &= \mathbb{P}_{\leq \lambda} \alpha(x) \mathbb{P}_{\leq \lambda^2}\left(\beta_{\lambda}\right)(x) V(\lambda^6x) \notag  \\
    & \qquad + \mathbb{P}_{>\lambda} \alpha(x) \mathbb{P}_{\leq \lambda^2}\left(\beta_{\lambda}\right)(x) V(\lambda^6x) + \alpha(x) \mathbb{P}_{>  \lambda^2}\left(\beta_{\lambda}\right)(x) V(\lambda^6x) \, . \label{eq:freqdecomp-RL}
\end{align}
The first term in \eqref{eq:freqdecomp-RL} can now be estimated using Bernstein's inequality and the uniform bound on $\| \beta_\lambda\|_{L^1}$ by
\begin{align*}
    \| \mathbb{P}_{\leq \lambda} \alpha(x) \mathbb{P}_{\leq \lambda^2}\left(\beta_{\lambda}\right)(x) V(\lambda^6x) \|_{\dot H^{-2}} & \lesssim \lambda^{-6} \| \mathbb{P}_{\leq \lambda} \alpha(x) \mathbb{P}_{\leq \lambda^2}\left(\beta_{\lambda}\right)(x) V(\lambda^6x) \|_{L^2} \\
    &\lesssim \lambda^{-6} \| V \|_{L^\infty} \| \alpha \|_{L^\infty} \| \mathbb{P}_{\leq \lambda^2} \beta_\lambda \|_{L^2} \\
    &\lesssim \lambda^{-6} \| V \|_{L^\infty} \| \alpha \|_{L^\infty} \lambda^2 \| \mathbb{P}_{\leq \lambda^2} \beta_\lambda \|_{L^1} \\
    &\rightarrow 0
\end{align*}
as $\lambda\rightarrow \infty$.  Meanwhile, for the second term we use that $L^1$ embeds into $\dot H^{-2}$, the uniform $L^1$ bound on $\beta_\lambda$, and the smoothness of $\alpha$.  Finally, for the last term, we use the same embedding and the assumption that $\| \mathbb{P}_{> \lambda^2} \beta_\lambda \|_{L^1} \rightarrow 0$ as $\lambda \rightarrow \infty$.
\end{proof}

\section{Statement of Inductive Proposition and proof of Theorem~\ref{thm:main}}\label{Section3}

We make the following inductive assumptions about $u_q$ and $R_q$: 
\begin{enumerate}
    \item\label{i:ind:1} $(u_q,R_q)$ solves~\eqref{eq:RS}.
    \item\label{i:ind:2} $\fint_{\mathbb{T}^2} u_q dx = 0$ and $u_q \in C^{\infty}(\T^2)$, and there exists $C > 0$ such that 
    \[  \| u_q \|_{L^{p(q)}(\mathbb{T}^2)} > C^{-1}(1+2^{-q}) \, , \]
    where $p(q') = 2 - 2^{-q'-10}$ for any $q'$.  In addition, we set $u_{-1} = 0$ and assume that for $q' \leq q$ 
    \begin{equation}\label{e:w:est}
    \Vert u_{q'}-u_{q'-1} \Vert_{L^{p(q')}} < 2^{-q'-2} \, 
    \end{equation}
    and $u_{q'} - u_{q'-1}$ is qualitatively smooth.
    \item\label{i:ind:3} $R_q$ is $C^{\infty}$ and $\| R_q \|_{\dot{H}^{-2}} < 2^{-q-10}$.
    \item\label{i:ind:4} For each $q' \leq q$, there exists $j$ (depending on $q'$) such that $\mathbb{P}_{2^j}(u_{q'} - u_{q'-1}) = u_{q'} - u_{q'-1}$. In addition for all $q' < q'' \leq  q$, the frequency support of $u_{q''} - u_{q''-1}$ and $u_{q'} - u_{q'-1}$ are in disjoint Littlewood-Paley shells.
    \item\label{i:ind:5} With the same $C$ as item~\ref{i:ind:2}, we assume that
    \[ \sum_{\substack{m, n \leq q-1 \\ m \neq n}} \| (u_{m+1} - u_m) \otimes (u_{n+1} - u_n) \|_{L^1} < C-2^{-q} \]
    and
    \[ \sum_{m \leq q-1} \| \mathbb{P}_{\not = 0} \left((u_{m+1} -u_m) \otimes (u_{m+1} - u_m)\right) \|_{\dot{H}^{-2}} < C-2^{-q}. \]
\end{enumerate}

\begin{proposition}[\textbf{Inductive proposition}]\label{prop:ind} Fix $C > 0$. Assume for $q \in \mathbb{N}$ that $(u_q, R_q)$ solve \eqref{eq:RS} and satisfy items~\ref{i:ind:1}--\ref{i:ind:5} with this $C$. Then there exists $w_{q+1} = u_{q+1} - u_q$ and $R_{q+1}$, both qualitatively smooth and mean zero, such that items~\ref{i:ind:1}--\ref{i:ind:5} hold with $q$ replaced by $q+1$.
\end{proposition}

Assuming for the moment that the inductive proposition holds, we can prove the main result. We then prove Proposition~\ref{prop:ind} in the next section.
\begin{proof}[Proof of Theorem~\ref{thm:main} using Proposition~\ref{prop:ind}]
Set $u_0(x) = A\sin(2\pi x_2)e_1$, $p_0=0$, and
$$
R_0(x) =
\begin{bmatrix}
    0 & -2\pi A \cos(2\pi x_2)\\
    -2\pi A \cos(2\pi x_2) & 0\\
\end{bmatrix} \, . 
$$
Then $u_0$ is mean zero, $\operatorname{div} (u_0\otimes u_0)=0$, and $-\Delta u_0 = A(2\pi)^2 \sin(2\pi x_2)e_1 = \operatorname{div} (R_0)$, so that $(u_0, R_0)$ satisfy item~\ref{i:ind:1}.  We may choose $A$ small enough so that item~\ref{i:ind:3} holds as well.  Clearly item~\ref{i:ind:4} holds.  Finally, we may choose $C>0$ so that items~\ref{i:ind:2} and~\ref{i:ind:5} hold.

Now assume that the inductive proposition holds and fix $\epsilon,\epsilon' \in (0,1)$. Define $u=\lim_{q\to \infty} u_{q}$. We start by showing this limit exists in the $L^{2 - \epsilon'}$ sense for all $\epsilon'\in (0,1)$. Let $\tilde{q}$ be such that $p(\tilde{q}) \geq 2 - \epsilon'$. Hence from \ref{i:ind:2}, specifically \eqref{e:w:est} and the smoothness of $u_{q'}$ for all $q'$, we have
\begin{equation*}
    \begin{split}
        \sum_{q' \leq q} \Vert u_{q'} - u_{q'-1} \Vert_{L^{2 - \epsilon'}} &= \sum_{q' < \tilde{q}} \Vert u_{q'} - u_{q'-1} \Vert_{L^{2 - \epsilon'}} + \sum_{\tilde{q} \leq q' \leq q} \Vert u_{q'} - u_{q'-1} \Vert_{L^{2 - \epsilon'}}\\
        &< \sum_{q' < \tilde{q}} \Vert u_{q'} - u_{q'-1} \Vert_{L^{2 - \epsilon'}} + \sum_{\tilde{q} \leq q'} 2^{-q'-2}\\
        &\lesssim_{\epsilon'} 1
    \end{split}
\end{equation*}
where the implicit constant is independent of $q$ but depends on $\epsilon'$. Hence $u$ is well defined as an element of $L^{2-\epsilon'}$, and since our choice of $\epsilon'$ was arbitrary, $u$ is a well defined element of $\bigcap_{\epsilon' \in (0,1)} L^{2 - \epsilon'}$. Now set
$$
r = \frac{2\epsilon}{1 + \epsilon} \, . 
$$
Since $r \in (0,1)$, we see that $u \in L^{2 - r} = L^{\frac{2}{1 + \epsilon}}$. Sobolev embedding shows that $\dot{H}^{\epsilon} \subset L^{\frac{2}{1 - \epsilon}}$, hence $L^{\frac{2}{1 + \epsilon}} = \left(L^{\frac{2}{1 - \epsilon}}\right)^* \subset \dot{H}^{-\epsilon}$ and so $u \in \dot{H}^{-\epsilon}$. Since our choices of $\epsilon,\epsilon'$ were arbitrary, we conclude that
$$
u \in \bigcap_{\epsilon \in (0,1)} L^{2 - \epsilon} \cap \dot{H}^{-\epsilon}.
$$
From~\ref{i:ind:1} and~\ref{i:ind:2} we get that $u$ is not constant. Now for every smooth divergence free vector field $\phi$, we have using item~\ref{i:ind:1} that
\begin{equation}\label{eq:weak-form}
\int_{\mathbb{T}^{2}} \left(\partial_{i}\phi^{j}u_{q}^{i}u_{q}^{j}+\partial_{ii}\phi^{j}u_{q}^{j}\right) =\int_{\mathbb{T}^{2}}\partial_{i}\phi^{j}R_{q}^{ij}   \, .
\end{equation}
Taking $q\to \infty$, we obtain the following:
\begin{itemize}
    \item $R_{q}^{ij}\to 0$ as $q \to \infty$ in $\dot{H}^{-2}$ due to \ref{i:ind:3}. Therefore 
    \[ \int_{\T^2} \Delta g \Delta^{-1}R_{q}^{ij}\longrightarrow 0 \quad \forall g\in \dot{H}^{2}\]
    and so upon integrating by parts the right hand side of \eqref{eq:weak-form} tends to $0$ as $q \to \infty$.
    \item By the previous bullet point, we have that the left hand side of \eqref{eq:weak-form} tends to $0$. Since $u_{q}\to u$ in $L^{2-\varepsilon}(\T^{2})$, we have
    \[\int_{\mathbb{T}^{2}}\partial_{ii}\phi^{j}u_{q}^{j}\longrightarrow \int_{\T^{2}}\partial_{ii}\phi^{j}u^{j} \, . \]
    Finally using item~\ref{i:ind:4}, we have that
    \begin{align*}
    \lim_{q\to \infty}\int_{\T^2}\partial_{i}\phi^{j}u_{q}^{i}u_{q}^{j}
      &=  \lim_{q\to \infty} \int_{\mathbb{T}^{2}}\partial_{i}\phi^{j}\mathbb{P}_{\neq 0}(u_{q}^{i}u_{q}^{j})\\
        &=\lim_{q\to \infty} \int_{\T^{2}}\partial_{i}\phi^{j}\sum_{m,n\leq q-1}\mathbb{P}_{\neq 0}\lp (u_{m+1}-u_{m})\otimes (u_{n+1}-u_{n}) \rp\\
        &=\int_{\T^{2}}\partial_{i}\phi^{j}\sum_{m,n}\mathbb{P}_{\neq 0}\lp (u_{m+1}-u_{m})\otimes (u_{n+1}-u_{n}) \rp\\
        &=:\int_{\T^{2}}\partial_{i}\phi^{j}\mathbb{P}_{\not = 0}\left(u^iu^j\right)
         \end{align*}
\end{itemize}
Notice the fourth equality is defined this way using Definitions~\ref{def:paras} and~\ref{def:para:solns}, and this definition makes sense due to \ref{i:ind:5}, while the first  equality is due to $\phi$ being divergence free. Hence we obtain
$$
\left\langle -\Delta \phi^j, u^j \right\rangle_{\dot{H}^{\epsilon},\dot{H}^{-\epsilon}} - \left\langle \partial_i \phi^j, \mathbb{P}_{\not = 0}(u^iu^j)\right\rangle_{\dot{H}^2,\dot{H}^{-2}} = -\lim_{q \to \infty} \int_{\mathbb{T}^{2}} \left(\partial_{i}\phi^{j}u_{q}^{i}u_{q}^{j}+\partial_{ii}\phi^{j}u_{q}^{j}\right) = 0
$$
and so $u$ is a nontrivial solution to \eqref{eq: SNSE} in the sense of Definition \ref{def:para:solns}.
\end{proof}

\section{Proof of Proposition~\ref{prop:ind}}\label{Section4}
Throughout this section, we will utilize a parameter $\lambda_{q+1}$, which will be chosen as a large power of $2$. The constraints which enforce our large choice of $\lambda_{q+1}$ are those contained in Definition \ref{def:increm}, ~\eqref{l:one},~\eqref{eq:indprop2:est}, \eqref{eq:indprop3:est1}, \eqref{eq:indprop3:est2}, \eqref{eq:wc2:est}, \eqref{eq:wcwp:est}, \eqref{eq:est:oldRL}, \eqref{eq:est:newRL}, \eqref{eq:est:offdiag_wp3}, \eqref{eq:wp1wp3:est_final}, \eqref{eq:wp2wp3:est_final}, \eqref{eq:lambda_ineq}, \eqref{eq:lambdaineq}, and \eqref{eq:indprop5:est1}.  Throughout this section, implicit constants may appear but they are independent of $\lambda_{q+1}$.

\subsection{Construction of \texorpdfstring{$w_{q+1}$}{wq1}}

We begin by specifying our fundamental building blocks, which are the usual intermittent Mikado flows. Our presentation here follows~\cite[Lemma~2.3]{DS2017} and~\cite[Lemma~6.7]{BV2020}.
\begin{lemma}[\textbf{Intermittent Mikado flows}]\label{lem:boldW}
Let $\lambda_{q+1} \in \mathbb{N}$ be given and $\epsilon_\gamma\in (0,1)$ be given such that $\lambda_{q+1}^{\epsilon_\gamma} \in \mathbb{N}$. For each $k \in \Lambda$ from Lemma~\ref{lem:geom}, there exists smooth $\mathbb{W}_{q+1}^{k^{\perp}}$ such that
\begin{enumerate}
        \item\label{w:1} $\operatorname{div}(\mathbb{W}_{q+1}^{k^{\perp}})=0$ and $\operatorname{div}(\mathbb{W}_{q+1}^{k^{\perp}} \otimes \mathbb{W}_{q+1}^{k^{\perp}})=0$.
        \item\label{w:2} $\mathbb{W}_{q+1}^{k^{\perp}}(x)$ is parallel to $k^{\perp}$ for all $x\in \mathbb{T}^{2}$.
        \item\label{w:3} $\fint_{\mathbb{T}^{2}}\mathbb{W}_{q+1}^{k^{\perp}}\otimes \mathbb{W}_{q+1}^{k^{\perp}} =k^{\perp}\otimes k^{\perp}$.
        \item\label{w:4} $\int_{\mathbb{T}^{2}} \mathbb{W}_{q+1}^{k^{\perp}} =0$.
        \item\label{w:5} $\| \nabla^m \mathbb{W}_{q+1}^{k^{\perp}} \|_{L^p(\T^2)} \lesssim \lambda_{q+1}^{(\epsilon_\gamma-1)\left( \frac 1p - \frac 12 \right)} \lambda_{q+1}^m $.
        \item\label{w:6} $\mathbb{W}_{q+1}^{k^{\perp}}$ is $\left(\frac{\mathbb{T}}{\lambda_{q+1}^{\epsilon_\gamma}}\right)^2$-periodic.
        \item\label{w:7} $\lim_{\lambda_{q+1}\rightarrow \infty} \| \mathbb{P}_{> \lambda_{q+1}^2}(\mathbb{W}_{q+1}^{k^{\perp}} \otimes \mathbb{W}_{q+1}^{k^{\perp}}) \|_{L^1} =0$. 
\end{enumerate}
\end{lemma}

\begin{proof}
    Fix $k \in \Lambda$ and $\phi_{q+1} \in C^{\infty}(\R)$ with $\hat{\phi}_{q+1}(0) = 0$ and spatial support contained in $(0,1)$. Put
    $$
    \psi_{q+1}^{k^{\perp}}(x)k^{\perp} = \phi_{q+1}(k \cdot x) k^{\perp}.
    $$
    Now periodize $\lambda_{q+1}^{\frac{1-\epsilon}{2}} \psi_{q+1}^{k^{\perp}}(\lambda_{q+1}^{1-\epsilon}x)k^{\perp}$ so that it is $\Z^2$-periodic (this is possible since $5\Lambda \subset \Z^2$) and set this to be $\chi_{q+1}^{k^{\perp}}(x) k^{\perp}$. Finally set $\rho_{q+1}^{k^{\perp}}(x) = \chi_{q+1}^{k^{\perp}}(\lambda_{q+1}^{\epsilon}x)$ and then put
    $$
    \mathbb{W}_{q+1}^{k^{\perp}}(x) = \rho_{q+1}^{k^{\perp}}(x) k^{\perp}.
    $$
    We claim this will have all of the desired properties. Items~\ref{w:1},~\ref{w:2},~\ref{w:4}, and~\ref{w:6} are all immediate. For~\ref{w:3}, we simply modify our choice of $\phi_{q+1}$ so that the $L^2$ norm of $\rho_{q+1}^{k^{\perp}}$ is normalized. Then finally for~\ref{i:ind:5}, note that the support of $\mathbb{W}_{q+1}^{k^{\perp}}$ in $\mathbb{T}^2$ will consist of $\lambda_{q+1}^{\epsilon}$ parallelograms of length at most $\sqrt{2}$ and of width $\lambda_{q+1}^{-1}$. Hence
    $$
    \Vert \mathbb{W}_{q+1}^{k^{\perp}} \Vert_{L^p(\mathbb{T}^2)} \lesssim \lambda_{q+1}^{\frac{\epsilon}{p}} \lambda_{q+1}^{-\frac{1}{p}} \lambda_{q+1}^{\frac{1-\epsilon}{2}} = \lambda_{q+1}^{(\epsilon - 1)\left(\frac{1}{p} - \frac{1}{2}\right)}.
    $$
    Since the maximum frequency of $\mathbb{W}_{q+1}^{k^{\perp}}$ is (morally speaking) of order $\lambda_{q+1}$, derivatives correspond to multiplication by $\lambda_{q+1}$, up to implicit constants. This observation in addition to the previous computation gives~\ref{w:5}.  The last estimate follows from the fact that derivatives on $\mathbb{W}_{q+1}$ cost $\lambda_{q+1} \ll \lambda_{q+1}^2$ as $\lambda_{q+1} \rightarrow \infty$.
\end{proof}

Next, we define the velocity increment $w_{q+1}:=u_{q+1}-u_{q}$.
\begin{definition}[\textbf{Definition of $w_{q+1}$}]\label{def:increm}
    Consider the functions $\ga_{k}$ and the set $\Lambda$ from Lemma \ref{lem:geom} and $\mathbb{W}_{q+1}^{k^{\perp}}=\rho^{k^{\perp}}_{q+1}k^{\perp}$ from Lemma \ref{lem:boldW} with parameter choices $\lambda_{q+1}$ a very large power of $2$, and $\epsilon_{\gamma} \in (0,1)$ chosen such that $\lambda_{q+1}^{\epsilon_{\gamma}}$ is an integer and $\epsilon_\gamma< \epsilon$ from Lemma~\ref{lem:geom}. Let 
    \begin{equation}\label{eq:akRq}
        a_k(R_q) = \left(\epsilon_{\gamma}^{-1} \Vert R_q \Vert_{L^{\infty}}\right)^{\sfrac 12} \gamma_k\left(I - \frac{\epsilon_{\gamma} R_q}{\Vert R_q \Vert_{L^{\infty}}}\right).
    \end{equation}
    We define the increment
    \begin{equation} \label{eq:increm1}
        w_{q+1}:=\nabla^{\perp}\left( \frac{2}{5\pi} \sum_{k \in \Lambda} \mathbb{P}_{\leq \lambda_{q+1}}\left(a_k(R_q)\right) \mathbb{P}_{\leq \lambda_{q+1}^2}\left(\rho_{q+1}^{k^{\perp}}\right) \frac{\sin\left(\frac{5\pi}{2}\lambda_{q+1}^6 k \cdot x\right)}{\lambda_{q+1}^6}\right).
    \end{equation}
\end{definition}

Applying the product rule to the increment (\ref{eq:increm1}), we obtain the following decomposition of $w_{q+1}$:
\begin{equation} \notag
    w_{q+1}=w_{q+1}^{(c)}+w_{q+1}^{(p)},
\end{equation}
where
\begin{equation}\label{eq:corrector}
    w_{q+1}^{(c)} = \frac{2}{5\pi} \sum_{k \in \Lambda} \nabla^{\perp}\left(\mathbb{P}_{\leq \lambda_{q+1}}\left(a_k(R_q)\right) \mathbb{P}_{\leq \lambda_{q+1}^2}\left(\rho_{q+1}^{k^{\perp}}\right)\right) \frac{\sin\left(\frac{5\pi}{2}\lambda_{q+1}^6 k \cdot x\right)}{\lambda_{q+1}^6}
\end{equation}
and
\begin{equation}\label{eq:principal}
    w_{q+1}^{(p)} = \sum_{k \in \Lambda} \mathbb{P}_{\leq \lambda_{q+1}}\left(a_k(R_q)\right) \mathbb{P}_{\leq \lambda_{q+1}^2}\left(\rho_{q+1}^{k^{\perp}}\right) \cos\left(\frac{5\pi}{2}\lambda_{q+1}^6 k \cdot x\right) k^{\perp}.
\end{equation}    
We refer to $w_{q+1}^{(c)}$ as the \textit{corrector component} and $w_{q+1}^{(p)}$ as the \textit{principal component}. We now prove $L^p$ bounds for each of these components, as well as the increment itself.
\begin{lemma}[\textbf{Bounds for $w_{q+1}$}]
Let $w_{q+1}$, $w_{q+1}^{(c)}$, and $w_{q+1}^{(p)}$ be as in \eqref{eq:increm1}, \eqref{eq:corrector}, and \eqref{eq:principal}. Then for a sufficiently large choice of $\lambda_{q+1}$ and all $1 \leq p \leq \infty$,
\begin{equation}\label{eq:w^c:est}
        \Vert w_{q+1}^{(c)} \Vert_{L^p} \lesssim \lambda_{q+1}^{-5} 
    \end{equation}
    \begin{equation}\label{eq:w^p:est}
        \Vert w_{q+1}^{(p)} \Vert_{L^p} \lesssim \lambda_{q+1}^{(\epsilon_{\gamma} - 1)\left(\frac{1}{p} - \frac{1}{2}\right)}.
    \end{equation}
    Therefore we also have
    \begin{equation}\label{eq:w:est}
        \Vert w_{q+1} \Vert_{L^p} \lesssim \lambda_{q+1}^{-5} + \lambda_{q+1}^{(\epsilon_{\gamma} - 1)\left(\frac{1}{p} - \frac{1}{2}\right)}.
    \end{equation}
\end{lemma}
\begin{proof}
    Applying Lemma \ref{lem:proj} we have
    \begin{equation*}
    \begin{split}
        \Vert w_{q+1}^{(c)} \Vert_{L^{p}} &\leq \sum_{k \in \Lambda} \left\Vert \mathbb{P}_{\leq \lambda_{q+1}}\left(\nabla^{\perp} a_k(R_q)\right)\right\Vert_{L^{\infty}} \left\Vert \mathbb{P}_{\leq \lambda_{q+1}^2}\left(\rho_{q+1}^{k^{\perp}}\right)\right\Vert_{L^p} \left\Vert \frac{\sin\left(\frac{5\pi}{2}\lambda_{q+1}^6 k \cdot x\right)}{\lambda_{q+1}^6} \right\Vert_{L^{\infty}}\\
        &+ \sum_{k \in \Lambda} \left\Vert \mathbb{P}_{\leq \lambda_{q+1}}\left(a_k(R_q)\right) \right\Vert_{L^{\infty}} \left\Vert \mathbb{P}_{\leq \lambda_{q+1}^2}\left(\nabla^{\perp}\rho_{q+1}^{k^{\perp}}\right)\right\Vert_{L^{p}} \left\Vert \frac{\sin\left(\frac{5\pi}{2}\lambda_{q+1}^6 k \cdot x\right)}{\lambda_{q+1}^6}\right\Vert_{L^{\infty}}\\
        &\lesssim \lambda_{q+1}^{-5}
    \end{split}
\end{equation*}
and now applying both Lemma \ref{lem:proj} and Lemma \ref{lem:boldW} we have
\begin{equation*}
\begin{split}
    \Vert w_{q+1}^{(p)} \Vert_{L^{p}} &\lesssim \sum_{k \in \Lambda} \left\Vert \mathbb{P}_{\leq \lambda_{q+1}}\left(a_k(R_q)\right) \right\Vert_{L^{\infty}}  \left\Vert \mathbb{P}_{\leq \lambda_{q+1}^2}\left(\mathbb{W}_{q+1}^{k^{\perp}}\right)\right\Vert_{L^{p}} \left\Vert \cos\left(\frac{5\pi}{2}\lambda_{q+1}^6 k \cdot x\right) \right\Vert_{L^{\infty}}\\
    &\lesssim \lambda_{q+1}^{(\epsilon_{\gamma} - 1)\left(\frac{1}{p} - \frac{1}{2}\right)}.
    \end{split}
\end{equation*}
Then~\eqref{eq:w:est} follows immediately from these two estimates.
\end{proof}

\subsection{Proof of item~\ref{i:ind:1}}
Using the assumption $\operatorname{div}(u_q) = 0$ as well as \eqref{eq:increm1} we have
$$
\operatorname{div}(u_{q+1}) = \operatorname{div}(u_{q}) + \operatorname{div}(w_{q+1}) = 0
$$
and
\begin{equation*}
    \begin{split}
        \operatorname{div}(u_{q+1} \otimes u_{q+1}) &- \Delta u_{q+1} \\
        &= \operatorname{div}(w_{q+1} \otimes w_{q+1}) + \operatorname{div}(w_{q+1} \otimes u_q + u_q \otimes w_{q+1}) + \operatorname{div}(u_q \otimes u_q)\\
        &- \Delta w_{q+1} - \Delta u_q\\
        &= \operatorname{div}(w_{q+1} \otimes w_{q+1}) + \operatorname{div}(w_{q+1} \otimes u_q + u_q \otimes w_{q+1})\\
        &+ \operatorname{div}(R_q) - \nabla p_q + \Delta u_q - \Delta w_{q+1} - \Delta u_q\\
        &= \operatorname{div}(R_q + w_{q+1} \otimes w_{q+1}) - \nabla p_q + \operatorname{div}(w_{q+1} \otimes u_q + u_q \otimes w_{q+1}) - \Delta w_{q+1} \, .
    \end{split}
\end{equation*}
Setting
\begin{equation}\label{eq:Rq+1}
    R_{q+1} := \underbrace{R_q + w_{q+1} \otimes w_{q+1}}_{\textnormal{nonlinear error}} + \underbrace{w_{q+1} \otimes u_q + u_q \otimes w_{q+1}}_{\textnormal{Nash error}} -  \underbrace{(\nabla w_{q+1} + \nabla w_{q+1}^T)}_{\textnormal{dissipation error}}
\end{equation}
and using the identity
$$
\operatorname{div}\left(\nabla w_{q+1} + \nabla w_{q+1}^T\right) = \Delta w_{q+1} + \nabla(\operatorname{div}(w_{q+1})) = \Delta w_{q+1}
$$
shows that $(u_{q+1},R_{q+1})$ solves \eqref{eq:RS} with $p_{q+1} = p_q$,\footnote{Since we are not trying to show convergence of $u_q$ in $L^2$, the ``pressure increment'' which must be non-constant in standard intermittent Nash iterations for Euler or Navier-Stokes is a constant here.  This is visible in~\eqref{eq:akRq}, where we have used the rough normalization $\|R_q \|_{L^\infty}^{\sfrac 12}$.  Hence this method produces pressureless solutions if desired.} where $R_{q+1}$ is symmetric by inspection.

\subsection{Proof of item~\ref{i:ind:2} at level \texorpdfstring{$q+1$}{p1}}
Clearly
$$
\fint_{\mathbb{T}^2} u_{q+1}\, dx = \fint_{\mathbb{T}^2} w_{q+1}\, dx + \fint_{\mathbb{T}^2} u_q\, dx = 0
$$
since $u_{q}$ is assumed to have zero mean and $w_{q+1}$ is the perpendicular gradient of a smooth function. Note that~\eqref{e:w:est} is implied by the assumption that for all $q' \leq q$,
\begin{equation}\label{l:one}
\Vert w_{q'} \Vert_{L^{p(q')}} < 2^{-q'-2} \, .
\end{equation}
The desired inequality for $q' = q+1$ follows from \eqref{eq:w:est} and large enough $\lambda_{q+1}$ since $p(q+1) < 2$. Next, we need to show that
$$
C^{-1}(1 + 2^{-q-1}) < \Vert u_{q+1} \Vert_{L^{p(q+1)}} \, .
$$
Using the same bound at level $q$, we have
\begin{equation}\label{eq:indprop2:est}
    \begin{split}
        \Vert u_{q+1} \Vert_{L^{p(q+1)}} &\geq \Vert u_q - (-w_{q+1}) \Vert_{L^{p(q)}}\\
        &\geq \Vert u_q \Vert_{L^{p(q)}} - \Vert w_{q+1} \Vert_{L^{p(q)}}\\
        &\geq C^{-1}(1 + 2^{-q}) - \Vert w_{q+1} \Vert_{L^{p(q)}} \, .
    \end{split}
\end{equation}
Choosing $\lambda_{q+1}$ large enough such that $\Vert w_{q+1} \Vert_{L^{p(q)}} < C^{-1}2^{-q-1}$ will then be sufficient to give the lower bound. 

\subsection{Proof of item~\ref{i:ind:3}}
First, note that every term on the right hand side of \eqref{eq:Rq+1} is $C^{\infty}$ (by inductive assumption and by construction of $w_{q+1}$), so $R_{q+1}$ is as well. Now to show $\Vert R_{q+1} \Vert_{\dot{H}^{-2}} < 2^{-q-11}$, we analyze each term on the right hand side of \eqref{eq:Rq+1} separately.
\bigskip

\noindent\texttt{Nash error: } Using \eqref{eq:w:est} for $p = 1$ we compute 
\begin{align}
\|w_{q+1}\otimes u_{q} + u_q \otimes w_{q+1}\|_{\dot{H}^{-2}} &\lesssim \|w_{q+1}\otimes u_{q} + u_q \otimes w_{q+1}\|_{L^{1}} \notag \\
&\leq \|u_{q}\|_{L^{\infty}}\|w_{q+1}\|_{L^{1}} \notag \\
&\lesssim \lambda_{q+1}^{-5} + \lambda_{q+1}^{\frac{\epsilon_{\gamma}-1}{2}} \notag \\
&< 2^{-q-15} \, . \label{eq:indprop3:est1}
\end{align}
Note that we used the fact that $u_{q}$ is qualitatively smooth as well as the fact that $\lambda_{q+1}$ is large enough to ensure the final inequality, as well as the embedding $L^1 \subset \dot{H}^{-2}$.
\bigskip

\noindent\texttt{Dissipation error: } Again using \eqref{eq:w:est} for $p = 4/3$ we have
\begin{equation}\label{eq:indprop3:est2}
    \begin{split}
        \Vert \nabla w_{q+1} + \nabla w_{q+1}^T \Vert_{\dot{H}^{-2}}
        &\lesssim \Vert w_{q+1} \Vert_{\dot{H}^{-1}}\\
        &\lesssim \Vert w_{q+1} \Vert_{L^{4/3}}\\
        &\lesssim \lambda_{q+1}^{-5} +  \lambda_{q+1}^{\frac{\epsilon_{\gamma}-1}{4}}\\
        &< 2^{-q-15} \, .
    \end{split}
\end{equation}
Note that $\lambda_{q+1}$ must be large enough to ensure the final inequality.
\bigskip

\noindent\texttt{Nonlinear error:} We first write
\begin{equation}\label{eq:wq+1^2}
    w_{q+1} \otimes w_{q+1} = w_{q+1}^{(c)} \otimes w_{q+1}^{(c)} + w_{q+1}^{(p)} \otimes w_{q+1}^{(c)} + w_{q+1}^{(c)} \otimes w_{q+1}^{(p)} + w_{q+1}^{(p)} \otimes w_{q+1}^{(p)}.
\end{equation}
Any term in \eqref{eq:wq+1^2} which contains $w_{q+1}^{(c)}$ can be estimated easily using the strong decay property of $w_{q+1}^{(c)}$ from \eqref{eq:w^c:est}. For instance we have
\begin{equation}\label{eq:wc2:est}
    \Vert w_{q+1}^{(c)} \otimes w_{q+1}^{(c)} \Vert_{\dot{H}^{-2}} \lesssim \Vert w_{q+1}^{(c)} \otimes w_{q+1}^{(c)} \Vert_{L^\infty} \lesssim \Vert w_{q+1}^{(c)} \Vert_{L^{\infty}}^2 \lesssim \lambda_{q+1}^{-10} < 2^{-2q-100}
\end{equation}
and
\begin{equation}\label{eq:wcwp:est}
    \begin{split}
        \Vert w_{q+1}^{(p)} \otimes w_{q+1}^{(c)} + w_{q+1}^{(c)} \otimes w_{q+1}^{(p)} \Vert_{\dot{H}^{-2}} &\lesssim \Vert w_{q+1}^{(p)} \otimes w_{q+1}^{(c)} \Vert_{L^1}\\
        &\lesssim \Vert w_{q+1}^{(p)} \Vert_{L^2} \Vert w_{q+1}^{(c)} \Vert_{L^2}\\
        &\lesssim \lambda_{q+1}^{-5}\\
        &< 2^{-2q-100} \, ,
    \end{split}
\end{equation}
where we choose $\lambda_{q+1}$ large enough such that the final inequalities in \eqref{eq:wc2:est} and \eqref{eq:wcwp:est} hold. The final remaining term of \eqref{eq:wq+1^2} will require the following further decomposition of the principal component:
\begin{equation}\label{eq:wp-decomp}
    \begin{split}
        w_{q+1}^{(p)} &= - \sum_{k \in \Lambda} \mathbb{P}_{> \lambda_{q+1}}(a_k(R_q)) \mathbb{P}_{\leq \lambda_{q+1}^2}\left(\mathbb{W}_{q+1}^{k^{\perp}}\right) \cos\left(\frac{5\pi}{2}\lambda_{q+1}^6 k\cdot x\right)\\
        &- \sum_{k \in \Lambda} a_k(R_q) \mathbb{P}_{> \lambda_{q+1}^2}\left(\mathbb{W}_{q+1}^{k^{\perp}}\right) \cos\left(\frac{5\pi}{2}\lambda_{q+1}^6 k\cdot x\right)\\
        &+ \sum_{k \in \Lambda} a_k(R_q) \mathbb{W}_{q+1}^{k^{\perp}} \cos\left(\frac{5\pi}{2}\lambda_{q+1}^6 k\cdot x\right)\\
        :=&  w_{q+1}^{(p,1)} + w_{q+1}^{(p,2)} + w_{q+1}^{(p,3)}
    \end{split}
\end{equation}
We start with $w^{(p,3)}_{q+1} \otimes w^{(p,3)}_{q+1}$. We have
\begin{equation}\label{eq:wp3^2}
\begin{split}
   w^{(p,3)}_{q+1} \otimes w^{(p,3)}_{q+1} &=  \sum_{k \in \Lambda} a_k^2(R_q) \mathbb{W}_{q+1}^{k^{\perp}} \otimes \mathbb{W}_{q+1}^{k^{\perp}} \cos^2\left(\frac{5\pi}{2}\lambda_{q+1}^6 k\cdot x\right)\\
   &+ \sum_{k_1 \not = k_2} a_{k_1}(R_q)a_{k_2}(R_q) \mathbb{W}_{q+1}^{k_1^{\perp}} \otimes \mathbb{W}_{q+1}^{k_2^{\perp}} \cos\left(\frac{5\pi}{2}\lambda_{q+1}^6 k_1\cdot x\right)\cos\left(\frac{5\pi}{2}\lambda_{q+1}^6 k_2\cdot x\right)
   \end{split}
\end{equation}
We apply a trigonometric identity to the first term on the right hand side of \eqref{eq:wp3^2} to get
\begin{equation}\label{eq:diagterm}
    \begin{split}
        \sum_{k \in \Lambda} a_k^2(R_q) \mathbb{W}_{q+1}^{k^{\perp}} \otimes \mathbb{W}_{q+1}^{k^{\perp}} \cos^2\left(\frac{5\pi}{2}\lambda_{q+1}^6 k\cdot x\right) &= \sum_{k \in \Lambda} \frac{1}{2} a_k^2(R_q) \mathbb{W}_{q+1}^{k^{\perp}} \otimes \mathbb{W}_{q+1}^{k^{\perp}}\\
        &+ \sum_{k \in \Lambda} \frac{1}{2}a_k^2(R_q) \mathbb{W}_{q+1}^{k^{\perp}} \otimes \mathbb{W}_{q+1}^{k^{\perp}} \cos\left(5\pi\lambda_{q+1}^6 k\cdot x\right)
    \end{split}
\end{equation}
From Lemma~\ref{lem:boldW} we have
\begin{equation}\label{eq:meanWq+1^2}
    \sum_{k \in \Lambda} \frac{1}{2} a_k^2(R_q) \mathbb{W}_{q+1}^{k^{\perp}} \otimes \mathbb{W}_{q+1}^{k^{\perp}} =  \sum_{k \in \Lambda} \frac{1}{2} a_k^2(R_q) k^{\perp} \otimes k^{\perp} + \sum_{k \in \Lambda} \frac{1}{2} a_k^2(R_q) \mathbb{P}_{\not = 0}\left(\mathbb{W}_{q+1}^{k^{\perp}} \otimes \mathbb{W}_{q+1}^{k^{\perp}}\right) \, ,
\end{equation}
then using \eqref{eq:akRq} and Lemma \ref{lem:geom} we can write \eqref{eq:meanWq+1^2} as
\begin{equation}
    \sum_{k \in \Lambda} \frac{1}{2} a_k^2(R_q) \mathbb{W}_{q+1}^{k^{\perp}} \otimes \mathbb{W}_{q+1}^{k^{\perp}} = \epsilon_{\gamma}^{-1}\Vert R_q \Vert_{L^{\infty}}I - R_q + \sum_{k \in \Lambda} \frac{1}{2} a_k^2(R_q) \mathbb{P}_{\not = 0}\left(\mathbb{W}_{q+1}^{k^{\perp}} \otimes \mathbb{W}_{q+1}^{k^{\perp}}\right) \, .
\end{equation}
Therefore adding $R_q$ to~\eqref{eq:diagterm}, we obtain
\begin{equation}
\begin{split}
    R_q + \eqref{eq:diagterm}&= R_q + \sum_{k \in \Lambda} a_k^2(R_q) \mathbb{W}_{q+1}^{k^{\perp}} \otimes \mathbb{W}_{q+1}^{k^{\perp}} \cos^2\left(\frac{5\pi}{2}\lambda_{q+1}^6 k\cdot x\right) \\
    &= \epsilon_{\gamma}^{-1}\Vert R_q \Vert_{L^{\infty}}I + \sum_{k \in \Lambda} \frac{1}{2} a_k^2(R_q) \mathbb{P}_{\not = 0}\left(\mathbb{W}_{q+1}^{k^{\perp}} \otimes \mathbb{W}_{q+1}^{k^{\perp}}\right)\\
    &+ \sum_{k \in \Lambda} \frac{1}{2}a_k^2(R_q) \mathbb{W}_{q+1}^{k^{\perp}} \otimes \mathbb{W}_{q+1}^{k^{\perp}} \cos\left(5\pi\lambda_{q+1}^6 k \cdot x\right)
    \end{split}
\end{equation}
Since $\epsilon_{\gamma}^{-1} \Vert R_q \Vert_{L^{\infty}}I$ is constant, then upon taking the $\dot{H}^{-2}$ norm and applying triangle inequality this term vanishes. Now using Lemma~\ref{lem:RL2}  we may choose $\lambda_{q+1}$ large enough to ensure that
\begin{equation}\label{eq:est:oldRL}
    \left\Vert \sum_{k \in \Lambda} \frac{1}{2} a_k^2(R_q) \mathbb{P}_{\not = 0}\left(\mathbb{W}_{q+1}^{k^{\perp}} \otimes \mathbb{W}_{q+1}^{k^{\perp}}\right) \right\Vert_{\dot{H}^{-2}} < 2^{-2q-100} \, ,
\end{equation}
and similarly by applying Lemma \ref{lem:RL} we get for large enough $\lambda_{q+1}$ that
\begin{equation}\label{eq:est:newRL}
    \left\Vert \sum_{k \in \Lambda} \frac{1}{2}a_k^2(R_q) \mathbb{W}_{q+1}^{k^{\perp}} \otimes \mathbb{W}_{q+1}^{k^{\perp}} \cos\left(5\pi\lambda_{q+1}^6 k \cdot x\right) \right\Vert_{\dot{H}^{-2}} < 2^{-2q-100}
\end{equation}
Now we return to examine the second term in \eqref{eq:wp3^2}.\footnote{This bound follows as Lemma 5.2 in \cite{BBV2020} or Remark 5.4 in \cite{BN2023}.} Using the fact the size of the support of each $\mathbb{W}_{q+1}^{k^{\perp}}$ is $\lambda_{q+1}^{\epsilon_{\gamma}-1}$, the orthogonality of the support sets, and Lemma \ref{lem:boldW} we get
\begin{equation}\label{eq:est:offdiag_wp3}
    \begin{split}
        &\left\Vert \sum_{k_1 \not = k_2} a_{k_1}(R_q)a_{k_2}(R_q) \mathbb{W}_{q+1}^{k_1^{\perp}} \otimes \mathbb{W}_{q+1}^{k_2^{\perp}} \cos\left(\frac{5\pi}{2}\lambda_{q+1}^6 k_1\cdot x\right)\cos\left(\frac{5\pi}{2}\lambda_{q+1}^6 k_2\cdot x\right) \right\Vert_{\dot{H}^{-2}}\\
        \lesssim& \sum_{k_1 \not = k_2} \left\Vert  a_{k_1}(R_q)a_{k_2}(R_q) \mathbb{W}_{q+1}^{k_1^{\perp}} \otimes \mathbb{W}_{q+1}^{k_2^{\perp}} \cos\left(\frac{5\pi}{2}\lambda_{q+1}^6 k_1\cdot x\right)\cos\left(\frac{5\pi}{2}\lambda_{q+1}^6 k_2\cdot x\right) \right\Vert_{L^1}\\
        \lesssim& \sum_{k_1 \not = k_2} \left\Vert \mathbb{W}_{q+1}^{k_1^{\perp}} \otimes \mathbb{W}_{q+1}^{k_2^{\perp}} \right\Vert_{L^1}\\
        \lesssim& \sum_{k_1 \not = k_2} \left|\operatorname{supp}\left(W_{q+1}^{k_1^{\perp}}\right)\right| \left|\operatorname{supp}\left(W_{q+1}^{k_2^{\perp}}\right)\right| \left\Vert W_{q+1}^{k_1^{\perp}} \right\Vert_{L^{\infty}} \left\Vert W_{q+1}^{k_2^{\perp}} \right\Vert_{L^{\infty}}\\ 
        \lesssim& \left(\lambda_{q+1}^{\epsilon_{\gamma} - 1}\right)^2\left(\lambda_{q+1}^{\frac{1 - \epsilon_{\gamma}}{2}}\right)^2\\
        =& \lambda_{q+1}^{\epsilon_{\gamma} - 1}\\
        <& 2^{-2q - 100}.
    \end{split}
\end{equation}
So combining \eqref{eq:wp3^2}, \eqref{eq:est:oldRL}, \eqref{eq:est:newRL}, and \eqref{eq:est:offdiag_wp3} we arrive at
\begin{equation}\label{eq:est:wp3^2}
    \left\Vert R_q + w^{(p,3)}_{q+1} \otimes w^{(p,3)}_{q+1} \right\Vert_{\dot{H}^{-2}} < 2^{-2q-98} \, .
\end{equation}
Now we estimate the rest of the terms in \eqref{eq:wp-decomp}. As was the case when estimating \eqref{eq:wc2:est} and \eqref{eq:wcwp:est}, we will utilize the strong estimates on $w_{q+1}^{(p,1)}$ and $w_{q+1}^{(p,2)}$ to deduce the desired inequalities. For this exact reason, it suffices to only prove the desired inequalities for $w^{(p,1)}_{q+1} \otimes w^{(p,3)}_{q+1}$ and $w^{(p,2)}_{q+1} \otimes w^{(p,3)}_{q+1}$ since all remaining terms are at least as small as these two terms. So we have 
\begin{equation}\label{eq:wp1wp3:est}
    \begin{split}
        \Vert w^{(p,1)}_{q+1} \otimes w^{(p,3)}_{q+1} \Vert_{\dot{H}^{-2}} &\lesssim \sum_{k_1,k_2 \in \Lambda} \left\Vert \mathbb{P}_{> \lambda_{q+1}}(a_{k_1}(R_q)) a_{k_2}(R_q) \mathbb{P}_{\leq \lambda_{q+1}^2}\left(\mathbb{W}_{q+1}^{{k_1}^{\perp}}\right) \otimes \mathbb{W}_{q+1}^{{k_2}^{\perp}}\right\Vert_{L^1}\\
        &\leq \sum_{k_1,k_2 \in \Lambda} \left\Vert \mathbb{P}_{> \lambda_{q+1}}(a_{k_1}(R_q)) \right\Vert_{L^{\infty}} \left\Vert a_{k_2}(R_q) \right\Vert_{L^{\infty}} \left\Vert \mathbb{P}_{\leq \lambda_{q+1}^2}\left(\mathbb{W}_{q+1}^{{k_1}^{\perp}}\right) \right\Vert_{L^{2}} \left\Vert \mathbb{W}_{q+1}^{{k_2}^{\perp}}\right\Vert_{L^{2}}\\
        &\lesssim \sum_{k_1 \in \Lambda} \left\Vert \mathbb{P}_{> \lambda_{q+1}}(a_{k_1}(R_q)) \right\Vert_{L^\infty} \, .
    \end{split}
\end{equation}
Now since $a_{k_1}(R_q)$ is smooth, \eqref{eq:wp1wp3:est} tends to $0$ as we send $\lambda_{q+1} \to \infty$. Thus taking $\lambda_{q+1}$ large enough we have
\begin{equation}\label{eq:wp1wp3:est_final}
    \Vert w^{(p,1)}_{q+1} \otimes w^{(p,3)}_{q+1} \Vert_{\dot{H}^{-2}} < 2^{-2q-100}.
\end{equation}

Similarly we estimate
\begin{equation}
    \begin{split}
        \Vert w^{(p,2)}_{q+1} \otimes w^{(p,3)}_{q+1} \Vert_{\dot{H}^{-2}} &\lesssim \sum_{k_1,k_2 \in \Lambda} \left\Vert a_{k_1}(R_q) a_{k_2}(R_q) \mathbb{P}_{> \lambda_{q+1}^2}\left(\mathbb{W}_{q+1}^{{k_1}^{\perp}}\right) \otimes \mathbb{W}_{q+1}^{{k_2}^{\perp}} \right\Vert_{L^1}\\
        &\lesssim \sum_{k_1,k_2 \in \Lambda} \left\Vert a_{k_1}(R_q) a_{k_2}(R_q) \right\Vert_{L^{\infty}} \left\Vert \mathbb{P}_{> \lambda_{q+1}^2}\left(\mathbb{W}_{q+1}^{{k_1}^{\perp}}\right) \right\Vert_{L^2} \left\Vert \mathbb{W}_{q+1}^{{k_2}^{\perp}} \right\Vert_{L^2}\\
        &\lesssim \sum_{k_1 \in \Lambda} \left\Vert \mathbb{P}_{> \lambda_{q+1}^2}\left(\mathbb{W}_{q+1}^{{k_1}^{\perp}}\right) \right\Vert_{L^2} \, .
    \end{split}
\end{equation}
Now using~\eqref{w:7}, we deduce
\begin{equation}\label{eq:wp2wp3:est_final}
    \Vert w^{(p,2)}_{q+1} \otimes w^{(p,3)}_{q+1} \Vert_{\dot{H}^{-2}} < 2^{-2q-100}.
\end{equation}
So combining \eqref{eq:wc2:est}, \eqref{eq:wcwp:est}, \eqref{eq:est:wp3^2}, \eqref{eq:wp1wp3:est_final}, and \eqref{eq:wp2wp3:est_final}, we get
\begin{equation}\label{eq:IH3-final:est}
    \Vert R_q + w_{q+1} \otimes w_{q+1} \Vert_{\dot{H}^{-2}} < (32)2^{-2q-100} = 2^{-2q-95}.
\end{equation}

\bigskip
\noindent Finally from \eqref{eq:indprop3:est1}, \eqref{eq:indprop3:est2}, and \eqref{eq:IH3-final:est} we conclude
$$
\Vert R_{q+1} \Vert_{\dot{H}^{-2}} < (2)2^{-q-15} + 2^{-2q-95} < 2^{-q-11}.
$$

\subsection{Proof of item~\ref{i:ind:4} at level \texorpdfstring{$q+1$}{qp1}} From \eqref{eq:increm1}, it is clear that the frequency support of $w_{q+1}$ is contained in\footnote{The $\sfrac 54$ comes from the fact that we normalize the torus so that so that complex exponentials of the form $e^{2\pi i k\cdot x}$ for $k \in \mathbb{Z}^2$ are periodic.}
$$
\bigcup_{k \in \Lambda} B\left(\frac{5}{4}\lambda_{q+1}^6k, \lambda_{q+1}^2 + \lambda_{q+1}\right) \, ,
$$
since the the frequency support of the first two terms appearing inside the summation are contained in a ball centered at the origin of radius $\lambda_{q+1}^2 + \lambda_{q+1}$ and then the modulation by $\frac{5}{4}\lambda_{q+1}^6k$ converts to a translation by this same amount on the frequency side. Thus to show $\mathbb{P}_{2^j}(w_{q+1}) = w_{q+1}$ for some $j$, it suffices to prove that for sufficiently large $\lambda_{q+1}$ and any $k \in \Lambda$, there is $j$ such that
\begin{equation}\label{eq:ballincl}
    B\left(\frac{5}{4}\lambda_{q+1}^6k, \lambda_{q+1}^2 + \lambda_{q+1}\right) \subset B\left(\frac{5}{4}\lambda_{q+1}^6k, 2\lambda_{q+1}^2\right) \subset B\left(0,\frac{3}{2} \cdot 2^{j}\right) \setminus B(0,2^j)
\end{equation}
since $\varphi_j(\xi) = 1$ for $2^j \leq |\xi| \leq (3/2)2^j$ (see Definition \ref{def:projs}). The first inclusion of \eqref{eq:ballincl} is obvious since we are assuming $\lambda_{q+1}$ is large, in particular larger than $2$. So it remains to prove the second inclusion. Clearly $\left(\frac{5}{4}\lambda_{q+1}^6 + 2\lambda_{q+1}^2\right)k$ is the point in $\partial B\left(\frac{5}{4}\lambda_{q+1}^6k, 2\lambda_{q+1}^2\right)$ furthest from the origin and similarly $\left(\frac{5}{4}\lambda_{q+1}^6 - 2\lambda_{q+1}^2\right)k$ is the point in $\partial B\left(\frac{5}{4}\lambda_{q+1}^6k, 2\lambda_{q+1}^2\right)$ closest to the origin. Therefore since $|k| = 1$, \eqref{eq:ballincl} will be proven if we can demonstrate the existence of $j$ such that
\begin{equation}\label{eq:radii-ineq}
    2^j \leq \frac{5}{4}\lambda_{q+1}^6 - 2\lambda_{q+1}^2 < \frac{5}{4}\lambda_{q+1}^6 + 2\lambda_{q+1}^2 < \frac{3}{2} \cdot 2^{j}.
\end{equation}
Since $\lambda_{q+1}$ is a power of $2$, observe for large enough $\lambda_{q+1}$ we have
\begin{equation}\label{eq:lambda_ineq}
\lambda_{q+1}^6 \leq \frac{5}{4}\lambda_{q+1}^6 - 2\lambda_{q+1}^2 < \frac{5}{4}\lambda_{q+1}^6 + 2\lambda_{q+1}^2 < \frac{3}{2}\lambda_{q+1}^6.
\end{equation}
Hence this proves \eqref{eq:radii-ineq} for $j = 6\log_2(\lambda_{q+1})$ and large enough $\lambda_{q+1}$ and thus \eqref{eq:ballincl} also follows. The disjointness of the Littlewood-Paley shells follows from the fact we can assume without loss of generality that
\begin{equation}\label{eq:lambdaineq}
    2^{100}\lambda_{q'} < \lambda_{q'+1} \quad \text{for all } q' \leq q.
\end{equation}

\subsection{Proof of item~\ref{i:ind:5} at level \texorpdfstring{$q+1$}{qp1}} Recalling that
$$
\sum_{\substack{m, n \leq q-1 \\ m \neq n}} \| w_{m+1} \otimes w_{n+1} \|_{L^1} < C-2^{-q}
$$
by the inductive assumption, we have
\begin{equation*}
\begin{split}
    \sum_{\substack{m, n \leq q \\ m \neq n}} \| w_{m+1} \otimes w_{n+1} \|_{L^1} &= \sum_{\substack{m, n \leq q-1 \\ m \neq n}} \| w_{m+1} \otimes w_{n+1} \|_{L^1} + 2\sum_{n \leq q-1} \| w_{q+1} \otimes w_{n+1} \|_{L^1}\\
    &< C - 2^{-q} + 2 \Vert w_{q+1} \Vert_{L^1} \sum_{n \leq q - 1} \Vert w_{n+1} \Vert_{L^{\infty}}\\
    &\lesssim C - 2^{-q} + 2 \left(\sum_{n \leq q - 1} \Vert w_{n+1} \Vert_{L^{\infty}}\right) \left(\lambda_{q+1}^{\frac{\epsilon_{\gamma} - 1}{2}} + \lambda_{q+1}^{-5}\right) \, .
    \end{split}
\end{equation*}
Choosing $\lambda_{q+1}$ large enough such that
\begin{equation}\label{eq:indprop5:est1}
2 \left(\sum_{n \leq q - 1} \Vert w_{n+1} \Vert_{L^{\infty}}\right) \left(\lambda_{q+1}^{\frac{\epsilon_{\gamma} - 1}{2}} + \lambda_{q+1}^{-5}\right) < 2^{-q-10}
\end{equation}
gives
$$
\sum_{\substack{m, n \leq q \\ m \neq n}} \| w_{m+1} \otimes w_{n+1} \|_{L^1} < C - 2^{-q-1} \, ,
$$
which proves the first claim. For the second, we have by the inductive assumption that
$$
\sum_{m \leq q-1} \| w_{m+1} \otimes w_{m+1} \|_{\dot{H}^{-2}} = \sum_{m \leq q-1} \| \mathbb{P}_{\neq 0} \left( w_{m+1} \otimes w_{m+1}\right) \|_{\dot{H}^{-2}} < C-2^{-q}.
$$
So using \eqref{eq:IH3-final:est} and item~\ref{i:ind:3} we have
\begin{equation}
\begin{split}
    \sum_{m \leq q} \| w_{m+1} \otimes w_{m+1} \|_{\dot{H}^{-2}} &= \sum_{m \leq q-1} \| w_{m+1} \otimes w_{m+1} \|_{\dot{H}^{-2}} + \| w_{q+1} \otimes w_{q+1} \|_{\dot{H}^{-2}}\\
    &\leq C - 2^{-q} + \Vert R_{q} + w_{q+1} \otimes w_{q+1} \Vert_{\dot{H}^{-2}} + \Vert R_{q} \Vert_{\dot{H}^{-2}}\\
    &< C - 2^{-q} + 2^{-q-9}\\
    &< C - 2^{-q-1} \, .
\end{split}
\end{equation}
This completes the proof of the induction proposition.

\noindent\textsc{Department of Mathematics, Purdue University, West Lafayette, IN, USA.}
\vspace{.03in}
\newline\noindent\textit{Email address}: \href{mailto:bharga37@purdue.edu}{bharga37@purdue.edu}.
\newline\noindent\textit{Email address}: \href{mailto:eashkari@purdue.edu}{eashkari@purdue.edu}.
\newline\noindent\textit{Email address}: \href{mailto:ngismond@purdue.edu}{ngismond@purdue.edu}.
\newline\noindent\textit{Email address}: \href{mailto:mdnovack@purdue.edu}{mdnovack@purdue.edu}.

\end{document}